\documentclass[12pt,twoside]{amsart}
\usepackage{amssymb,amscd,amsxtra,calc}
\usepackage{amsmath}
\usepackage{xypic}
\usepackage[graph]{xy}
\usepackage{supertabular}
\usepackage{longtable}

\newtheorem{theorem}{Theorem}

\newtheorem{lemma}[theorem]{Lemma}

\newtheorem{remark}[theorem]{Remark}

\newtheorem{Claim}[theorem]{Claim}

\newtheorem*{ack}{Acknowledgments}

\title[splitting criterion for 2-bundles on Hirzebruch surfaces]{A cohomological splitting criterion for rank 2 vector bundles on Hirzebruch surfaces}
\author{Kazunori Yasutake}
\date{\today}
\address{Organization for the Strategic Coordination of Research and 
Intellectual Properties, Meiji University, Kanagawa 214-8571 Japan}
\email{tz13008@meiji.ac.jp}

\begin{document}

\maketitle

\begin{abstract}
In this note, we give a cohomological characterization of all rank 2 split vector bundles on Hirzebruch surfaces.
\end{abstract}

\section{Introduction}
Throughout this paper we work over an algebraic closed field $k$.
A vector bundle on a smooth projective variety is called $\it{split}$ if it is decomposed into a direct sum of line bundles.
Recently, Fulger and Marchitan $($\cite{fulger}$)$ obtained a cohomological characterization of some rank 2 split vector bundles on Hirzebruch surfaces over the complex number field , by using Buchdahl's Beilinson type spectral sequence $($\cite{buchdahl}$)$. 
However, it seems difficult to apply their argument to general cases.
The purpose of this paper is to give a simple characterization of all rank 2 split vector bundles on Hirzebruch surfaces by cohomological informations in arbitrary characteristic. 

\begin{theorem}
Let $\mathcal {E}$ be a $\mathrm{rank}$ $2$ vector bundle on a Hirzebruch surface $\mathbb{F}_n=\mathbb{P}_{\mathbb{P}^1}(\mathcal{O}_{\mathbb{P}^1}\oplus\mathcal{O}_{\mathbb{P}^1}(n))$.
Let $\mathcal{F}$ be a split rank $2$ vector bundle on $\mathbb{F}_n$.
If $\dim_k H^i(\mathcal{E}\otimes\mathcal{L})=\dim_k H^i(\mathcal{F}\otimes\mathcal{L})$ for any $0\leqslant i \leqslant 2$ and any line bundle $\mathcal{L}$ on $\mathbb{F}_n$,
then $\mathcal{E}$ is isomorphic to $\mathcal{F}$.
\end{theorem} 

It seems that our assumption of Theorem $1$ is stronger than the one of Theorem in \cite{fulger}.
However, if we know the Chern classes of $\mathcal{E}$, we need only a few assumptions $($see Lemma $2)$.
In the cases treated in  \cite{fulger}, the assumption of Lemma $2$ is essentially equivalent to the one of the Theorem of \cite{fulger}.    

\section*{Notation}

For a smooth projective variety $X$ and a vector bundle $\mathcal{E}$ on $X$, let $\mathbb{P}_X(\mathcal{E})$ be the projectivization of $\mathcal{E}$ in the sense of Grothendieck.
We denote $\dim_k H^i(\mathcal{E})$ by $h^i(\mathcal{E})$.
We denote the Hirzebruch surface $\mathbb{P}_{\mathbb{P}^1}(\mathcal{O}_{\mathbb{P}^1}\oplus\mathcal{O}_{\mathbb{P}^1}(n))$ by $\mathbb{F}_n$, the natural projection $\mathbb{F}_n\rightarrow\mathbb{P}^1$ by $\pi$, the minimal section on $\mathbb{F}_n$ by $\sigma$ $($i.e. $\sigma\cong\mathbb{P}^1$, $\sigma^2=-n)$ and a fiber of $\pi$ by $f$.
It is well-known that $\mathrm{Pic}(\mathbb{F}_n)=\mathbb{Z}\sigma\oplus\mathbb{Z}f$.

\section{Proof of Theorem 1}
To give the proof of Theorem 1, we show the following lemma.
\begin{lemma}
\hspace{1mm}
Let $\mathcal{E}$ be a $\mathrm{rank}$ $2$ vector bundle on $\mathbb{F}_n$.
Assume that $h^0(\mathcal{E}(-\sigma))=h^0(\mathcal{E}(-f))=0$, $h^0(\mathcal{E})\geqslant 1$ and $c_2(\mathcal{E})=0$.

\begin{enumerate}
\item Assume that $c_1(\mathcal{E})=-a\sigma-bf$, where $a$ and $b$ are nonnegative integers,
and that $h^0(\mathcal{E}(a\sigma+bf))\geqslant1+h^0(\mathcal{O}_{\mathbb{F}_n}(a\sigma+bf))$.
Then $\mathcal{E}$ is isomorphic to $\mathcal{O}_{\mathbb{F}_n}\oplus \mathcal{O}_{\mathbb{F}_n}(-a\sigma-bf)$.  
\item Assume that $c_1(\mathcal{E})=a\sigma-bf$, where $a$ and $b$ are integers such that $ab>0$,
and that $h^0(\mathcal{E}(-a\sigma+bf))\geqslant1+h^0(\mathcal{O}_{\mathbb{F}_n}(-a\sigma+bf))$
Then $\mathcal{E}$ is isomorphic to $\mathcal{O}_{\mathbb{F}_n}\oplus \mathcal{O}_{\mathbb{F}_n}(a\sigma-bf)$.  
\end{enumerate}
\end{lemma}

\begin{proof}
\hspace{1mm}
Put $\mathcal{L}=\mathcal{O}_{\mathbb{F}_n}(-a\sigma-bf)$ in Case 1 and  $\mathcal{L}=\mathcal{O}_{\mathbb{F}_n}(a\sigma-bf)$ in Case 2.
Since $h^0(\mathcal{E})\not=0$, we can take a nonzero section $0\not=s\in \mathrm{H}^0(\mathcal{E})$.
Put $Z:=(s=0)$.
If $s$ takes zero on some nonzero effective divisor $D>0$, we have a nonzero section of $\mathrm{H}^0(\mathcal{E}(-D))$.
This is a contradiction because $h^0(\mathcal{E}(-D))\leqslant\max \{h^0(\mathcal{E}(-\sigma)), h^0(\mathcal{E}(-f))\}=0$.
Therefore, $Z$ is of codimension 2. 
Since $c_2(\mathcal{E})=0$, we have $Z=\emptyset$.
Hence we obtain an exact sequence,
\[0\rightarrow \mathcal{O}_{\mathbb{F}_n}\xrightarrow{s}\mathcal{E}\rightarrow\mathcal{L}\rightarrow 0,\] 
since $c_1(\mathcal{E})=-a\sigma-bf$ in Case 1 and $c_1(\mathcal{E})=a\sigma-bf$ in Case 2.
We show that the above exact sequence splits.
Consider the long exact sequence ;
\[0\rightarrow \mathrm{Hom}(\mathcal{L},\mathcal{O}_{\mathbb{F}_n})\rightarrow \mathrm{Hom}(\mathcal{L},\mathcal{E})\rightarrow \mathrm{Hom}(\mathcal{L}, \mathcal{L})
\rightarrow\mathrm{Ext}^1(\mathcal{L}, \mathcal{O}_{\mathbb{F}_n})\rightarrow \cdots.\]
By the assumption, we have 
\begin{eqnarray*}
&&\dim_k \mathrm{Hom}(\mathcal{L},\mathcal{E})=h^0(\mathcal{E}\otimes\mathcal{L}^{\vee})\\
&&\geqslant h^0(\mathcal{L}^{\vee})+h^0(\mathcal{O}_{\mathbb{F}_n})=\dim_k \mathrm{Hom}(\mathcal{L},\mathcal{O}_{\mathbb{F}_n})+\dim_k \mathrm{Hom}(\mathcal{L}, 
\mathcal{L}).
\end{eqnarray*}
Therefore, the homomorphism $\mathrm{Hom}(\mathcal{L},\mathcal{E})\rightarrow \mathrm{Hom}(\mathcal{L}, \mathcal{L})$ is surjective.
Hence $\mathcal{E}$ is isomorphic to $\mathcal{O}_{\mathbb{F}_n}\oplus\mathcal{L}$.
\end{proof}

\begin{remark}
Let $\mathcal{E}$ be a split $\mathrm{rank}$ $2$ vector bundle on $\mathbb{F}_n$.
It is readily seen that there exists a line bundle $\mathcal{L}$ on $\mathbb{F}_n$ such that $c_1(\mathcal{E}\otimes\mathcal{L})=\alpha\sigma+\beta f$ with $\alpha, \beta\le 0$ or $\alpha\beta>0$ and that $c_2(\mathcal{E}\otimes\mathcal{L})=0$.
Therefore, by Lemma 2, we can characterize all split $\mathrm{rank}$ $2$ vector bundles on $\mathbb{F}_n$.   
\end{remark}

From now on, we give a proof of $\mathrm{Theorem}$ $1$.
We will begin with a proof of the following Claim.
\begin{Claim}\label{claim}
Under the assumptions of $\mathrm{Theorem}$ $1$, we have $\det(\mathcal{E})=\det(\mathcal{F})$ in $\mathrm{Pic}(\mathbb{F}_n)$ and $\deg c_2(\mathcal{E})=\deg c_2(\mathcal{F})$.
\end{Claim}
\begin{proof}
Take arbitrarily a very ample divisor $D$ on $\mathbb{F}_n$.
We may assume that $D$ is smooth.
By the assumptions, we have $\chi(\mathcal{E})=\chi(\mathcal{F})$ and $\chi(\mathcal{E}(-D))=\chi(\mathcal{F}(-D))$. 
Therefore, we obtain $\chi(\mathcal{E}|_D)=\chi(\mathcal{F}|_D)$. 
By Riemann-Roch theorem on the smooth curve $D$, we have $c_1(\mathcal{E})\cdot D=c_1(\mathcal{F})\cdot D$ $($cf. \cite{fulton}, Example 15.2.1.$)$.
Hence $c_1(\mathcal{E})$ is numerically equivalent to $c_1(\mathcal{F})$.
Because $\mathbb{F}_n$ is rational, we get $\det(\mathcal{E})=\det(\mathcal{F})$.
We also have $\deg c_2(\mathcal{E})=\deg c_2(\mathcal{F})$ from the Riemann-Roch theorem on $\mathbb{F}_n$ $($cf. \cite{fulton}, Example 15.2.2.$)$.
\end{proof}

Now we conclude the proof of Theorem 1.
By Remark 3, we may assume that $\mathcal{F}$ is isomorphic to $\mathcal{O}\oplus\mathcal{L}$, where $\mathcal{L}=\mathcal{O}_{\mathbb{F}_n}(-a\sigma-bf)$ with nonnegative integers $a,b\geqslant 0$ or $\mathcal{L}=\mathcal{O}_{\mathbb{F}_n}(a\sigma-bf)$ with positive integers $a,b> 0$.

By $\mathrm{Claim}$ $4$, we have $c_1(\mathcal{E})=c_1(\mathcal{F})=c_1(\mathcal{L})$ and $c_2(\mathcal{E})=0$.
By the assumptions of Theorem 1, we also have $h^0(\mathcal{E}(-\sigma))=h^0(\mathcal{F}(-\sigma))=0$, $h^0(\mathcal{E}(-f))=h^0(\mathcal{F}(-f))=0$, $h^0(\mathcal{E})=h^0(\mathcal{F})\geqslant 1$ and $h^0(\mathcal{E}\otimes\mathcal{L}^{\vee})=h^0(\mathcal{F}\otimes\mathcal{L}^{\vee})=1+h^0(\mathcal{L}^{\vee})$.
Then, by Lemma 2, we obtain the result of Theorem $1$.  
\\
\hspace{1mm}
\\
Similar arguments of the proof of Theorem 1 imply the following theorem.
\begin{theorem}
Let $S$ be a smooth projective surface with the Picard group $\mathrm{Pic}(S)\cong\mathbb{Z}$.  
Let $\mathcal{E}$ be a $\mathrm{rank}$ $2$ vector bundle on $S$.
Let $\mathcal{F}$ be a split $\mathrm{rank}$ $2$ vector bundle on $S$.
If $h^i(\mathcal{E}\otimes\mathcal{L})=h^i(\mathcal{F}\otimes\mathcal{L})$ for any $0\leqslant i \leqslant 2$ and any line bundle $\mathcal{L}$ on $S$,
then $\mathcal{E}$ is isomorphic to $\mathcal{F}$.
\end{theorem}

\begin{proof}
We may assume that $\mathcal{F}\cong\mathcal{O}_S\oplus\mathcal{M}$ where $\mathcal{M}$ is a line bundle on $S$ such that $\deg\mathcal{M}\leqslant0$.
In the same manner as in Claim \ref{claim}, we can verify that $\det(\mathcal{E})=\det(\mathcal{F})$ in $\mathrm{Pic}(S)$ since $\mathrm{Pic}(S)$ is isomorphic to $\mathbb{Z}$.
Moreover we also have $\deg c_2(\mathcal{E})=\deg c_2(\mathcal{F})$.
Therefore, we have an exact sequence 
\[0\rightarrow \mathcal{O}_{S}\rightarrow\mathcal{E}\rightarrow\mathcal{M}\rightarrow 0.\] 
Moreover, we can show that $\mathcal{E}\cong\mathcal{O}_S\oplus\mathcal{M}$ in the same way as in the proof of Theorem 1. 
\end{proof}

\begin{remark}
There are many surfaces having the Picard group $\mathrm{Pic}(S)\cong\mathbb{Z}$.
In fact, on the complex number field, it is known that a very general surface $S\subseteq\mathbb{P}^3$ of degree $d\geqslant4$ has the Picard group $\mathrm{Pic}(S)\cong\mathbb{Z}$. $($cf. \cite{harris}, Theorem$)$
\end{remark}

\begin{ack}
The author would like to express his gratitude to Professors Hajime Kaji, Yasunari Nagai, Yasuyuki Nagatomo, Eiichi Sato and Noriyuki Suwa for useful comments. 
\end{ack}

\end{document}